\documentclass[11pt]{amsart}
\usepackage{amsmath,amsfonts,amssymb,graphics,color,enumerate}
\usepackage[applemac]{inputenc}
\usepackage{hyperref}
\usepackage{stmaryrd}
\usepackage{amsthm}
\usepackage{mathtools}

\newtheorem{theorem}{Theorem}[section]
\newtheorem{lemma}[theorem]{Lemma}
\newtheorem{prop}[theorem]{Proposition}

\newtheorem{Definition}[theorem]{Definition}

\def\curl{\operatorname{curl}}
\def\div{\operatorname{div}}

\providecommand{\R}{\mathbb{R}}

\renewcommand{\leq}{\leqslant}

\numberwithin{equation}{section}

\begin{document}

\date{\today}
\title[Rigid bodies in $2$d Euler]{Dynamics of  rigid bodies in a two dimensional incompressible perfect fluid}

%
\author{Olivier Glass} 
\address{CEREMADE, UMR CNRS 7534, Universit\'e Paris-Dauphine, 
Place du Mar\'echal de Lattre de Tassigny, 75775 Paris Cedex 16, France} 
\author{Christophe Lacave} 
\address{Univ. Grenoble Alpes, CNRS, Institut Fourier, F-38000 Grenoble, France} 

\author{Alexandre Munnier} 
\address{Universit\'e de Lorraine, Institut Elie Cartan de Lorraine, UMR 7502, Nancy-Universit\'e, Vandoeuvre-l\`es-Nancy, F-54506, France
\& CNRS, Institut Elie Cartan de Lorraine, UMR 7502, Nancy-Universit\'e, Vandoeuvre-l\`es-Nancy, F-54506, France}
\author{Franck Sueur}
\address{Institut de Math\'ematiques de Bordeaux, UMR CNRS 5251,
Universit\'e de Bordeaux, 351 cours
de la Lib\'eration, F33405 Talence Cedex, France  $\&$ Institut  Universitaire de France}
%
%
\begin{abstract}
We consider the motion of several rigid bodies immersed in a two-dimensional incompressible perfect fluid, the whole system being bounded by an external impermeable fixed boundary. 
The fluid motion is described by the incompressible Euler equations and the motion of the rigid bodies is given by Newton's laws with forces due to the fluid pressure. 
We prove that, for smooth solutions, Newton's equations can be recast as a second-order ODE for the degrees of freedom of the rigid bodies with coefficients depending on the fluid vorticity and on the circulations around the bodies, but not anymore on the fluid pressure. 
This reformulation highlights geodesic aspects linked to the added mass effect, gyroscopic features generalizing the Kutta-Joukowski-type lift force, including body-body interactions through the potential flows induced by the bodies' motions, body-body interactions through the irrotational flows induced by the bodies' circulations, and interactions between the bodies and the fluid vorticity.
\end{abstract}
%
\maketitle
\section{Introduction}
We consider the motion of rigid bodies immersed in a two-dimensional incompressible perfect fluid. 
The whole system occupies a fixed impermeable bounded cavity $\Omega$ (a bounded simply connected open subset of $\R^2$).
At the initial time, the domain occupied by the rigid bodies are assumed to be disjoint non-empty regular connected and simply connected compact sets $\mathcal S_{ \kappa,0} \subset \Omega$, with $ \kappa \in \{1,2, \ldots , N \}$,
and
\begin{equation*}
\mathcal F_0 := \Omega \setminus \bigcup_{ \kappa \in \{1,2, \ldots , N \} } \, {\mathcal S}_{\kappa,0} , 
\end{equation*}
\par 
The rigid motion of the solid $\kappa$ is described at every moment by a rotation matrix
\begin{eqnarray*} 
R(\theta_\kappa (t)) := 
\begin{bmatrix}
\cos \theta_\kappa (t) & - \sin \theta_\kappa (t) \\
\sin \theta_\kappa (t) & \cos \theta_\kappa (t)
\end{bmatrix}, \quad \theta_\kappa (t) \in \R,
\end{eqnarray*}
and by the position $h_\kappa(t) \in \mathbb R^2$ of its center of mass  starting with the initial position $h_{ \kappa,0}$. 
The domain of the solid $\kappa$ at every time $t>0$ is therefore 
$$\mathcal S_\kappa(t) :=R(\theta_\kappa(t))(\mathcal S_{ \kappa,0}-h_{\kappa,0}) +h_\kappa(t), $$ 
while the domain of the fluid is 
$$\mathcal F(t) := \Omega\setminus \bigcup_{ \kappa \in \{1,2, \ldots , N \} } \, {\mathcal S_\kappa}(t). $$
The system is governed by the following set of coupled equations:
%
%
\begin{subequations} \label{SYS_full_system}
\begin{alignat}{3}
\nonumber \textbf{Fluid equations:} \\
\frac{\partial u}{\partial t}+(u\cdot\nabla)u +\nabla \pi&=0&\quad&\text{in }\mathcal F(t), \quad \text{for } t >0, \label{EEE1} \\
\div u&=0&&\text{in }\mathcal F(t), \quad \text{for } t >0, \label{E2} \\
u_{ \vert t=0}&=u_0&&\text{in }\mathcal F_0, \label{fluid_initial_cond} \\
\nonumber \textbf{Solids equations:} \\ \nonumber \quad \text{for any } \kappa \in \{1,2, \ldots , N \}, \\
m_\kappa h_\kappa'' (t)&=\int_{\partial\mathcal S_\kappa(t)}\pi n\, {\rm d}s,  \quad \text{for } t >0, \label{EqTrans} \\
\mathcal J_{\kappa} \theta_\kappa'' (t)&=\int_{\partial\mathcal S_\kappa(t)}(x-h_\kappa(t))^\perp\cdot \pi n\, {\rm d}s,  \quad \text{for } t >0, \label{EqRot} \\
 \Big( h_\kappa (0) , \theta_\kappa(0) , h'_\kappa (0) , \theta'_\kappa (0) \Big)
 &= \Big( h_{ \kappa,0}, 0, \ell_{ \kappa,0}, r_{ \kappa,0}
 \Big) , \\
\nonumber \textbf{Boundary conditions:} \\
\label{souslabis}
u\cdot n &=0 && \text{on } \partial\Omega,  \quad \text{for } t >0,\\
\nonumber
\quad \text{ and for any } \kappa \in \{1,2, \ldots , N \}, \\
\label{souslab}
u \cdot n &= \big( \theta'_\kappa (\cdot-h_\kappa)^\perp + h'_\kappa \big) \cdot n && \text{on }\partial\mathcal S_\kappa(t),  \quad \text{for } t >0 .
\end{alignat}
\end{subequations}
%
%
These are respectively the incompressible Euler equations, the Newton's balance law for linear and angular momenta of each body and 
impermeable boundary conditions. 
Above $u=(u_1,u_2)^t$ and $\pi$ denote the velocity and pressure fields in the fluid, $m_\kappa>0$ and $\mathcal{J}_\kappa>0$ denote respectively the mass and the moment of inertia of the body while the fluid is supposed to be homogeneous of density $1$, to simplify the notations. The exponent $t$ denotes the transpose of the vector. 
When $x=(x_1,x_2)^t$ the notation $x^\perp $ stands for $x^\perp =( -x_2 , x_1 )^t$, 
$n$ denotes the unit normal vector pointing outside the fluid, so that $n = \tau^\perp$, where $\tau$
denotes the unit counterclockwise tangential vector on $\partial \mathcal S_\kappa (t)$ and the unit clockwise tangential vector on $\partial \Omega$. 
Let us also emphasize that we will use ${\rm d}s$ as length element without any distinction on $\partial \Omega$,
$\partial \mathcal S(t)$ and on $\partial \mathcal S_{ \kappa,0}$. 
We will use the notations $\mathbf{q}_{\kappa}$ and $\mathbf{q}'_{\kappa}$ for vectors in $\R^{3}$ gathering both the linear and angular parts of the position and velocity:
$\mathbf{q}_{\kappa} := (h_{\kappa}^t,\theta_{\kappa})^t $ and
$ \mathbf{q}_{\kappa}':=(h^{\prime\, t}_{\kappa},\theta'_{\kappa})^t.$
The vectors $\mathbf{q}_{\kappa} $ and $\mathbf{q}'_{\kappa} $ are next concatenated into vectors of length $3N$:
$q = (\mathbf{q}_{1}^t , \ldots , \mathbf{q}_{N}^t)^t $ and ${q}' = (\mathbf{q}_{1}^{\prime\, t} , \ldots , \mathbf{q}^{\prime\, t}_{N} )^t, $
whose entries are relabeled respectively $q_k$ and $q'_k$ with $k$ ranging over $\{1,\ldots,3N\}$. Hence we have also:
$ q = (q_1, q_2,\ldots,q_{3N})^t $ and $ {q'} = (q'_1,q'_2,\ldots,q'_{3N})^t.$
Consequently, $k \in \{1,\dots,3N\}$ denotes the datum of both a solid number and a coordinate in $\{1,2,3\}$ so that $q_{k}$ and $q'_{k}$ denote respectively the coordinate of the position and of the velocity of a given solid. 
More precisely, for all $k \in \{1,\dots,3N\}$, we denote by $\llbracket k \rrbracket$ the quotient of the Euclidean division of $k-1$ by $3$,
$[k] = \llbracket k \rrbracket +1 \in \{1,\dots,N\}$ denotes the number of the solid and $(k) := k - 3\llbracket k \rrbracket \in \{1,2,3\}$ the considered coordinate. 
Throughout this paper we will not consider any collision, then we introduce the set of body positions without collision:
\[
\mathcal{Q} := \{q \in \R^{3N} \ :\ \min_{\kappa \neq \nu}{\rm d}(\mathcal{S}_{\kappa}(q), \Omega^c \cup \mathcal{S}_{\nu}(q)) > 0\},
\]
where $d$ is the Euclidean distance. 
The fluid domain is then totally described by $q\in \mathcal{Q}$: $\mathcal{F}(t)= \mathcal{F}(q(t))$.
\par
The fluid vorticity 
 $ \omega$ is defined by $ \omega := \curl u =\partial_1 u_2- \partial_2 u_1 $
and we set $\gamma := (\gamma_\kappa )_{\kappa=1,\ldots,N} $,
where $\gamma_\kappa := \int_{\partial\mathcal S_{\kappa}} u \cdot\tau \, {\rm d}s$ is the circulation around the rigid body $\mathcal S_{\kappa}$. 
 Let us recall that, for each $\kappa$, the circulation $\gamma_\kappa$ remains constant over time according to Kelvin's theorem.
 \ \par
To state our result let us introduce a few more notations.
We denote by $S^{++}_{3N} (\mathbb{R})$ the set of real symmetric positive-definite  $3N\times 3N$ matrices and 
by 
$\mathcal{BL}_s  (\mathbb{R}^{3N} \times \mathbb{R}^{3N} ; \mathbb{R}^{3N} )$ the space of bilinear symmetric mappings from $ \mathbb{R}^{3N} \times \mathbb{R}^{3N} $ to $ \mathbb{R}^{3N}$.
\begin{Definition} \label{Christq}
Given a $C^{\infty}$ mapping 
$ q\in \mathcal Q \mapsto {\mathcal M}(q) \in S^{++}_3 (\mathbb{R})$, 
we say that the $C^{\infty}$ mapping
$ q\in \mathcal Q \mapsto \Gamma (q) \in \mathcal{BL}_s  (\mathbb{R}^{3N} \times \mathbb{R}^{3N} ; \mathbb{R}^{3N} )$ is the connection associated with this mapping if for any 
$ p \in\mathbb R^{3N} $,
\begin{equation*}
\langle\Gamma (q) ,p,p\rangle :=\big(\sum_{1\leqslant i,j\leqslant 3N} (\Gamma (q) )^k_{i,j} p_i p_j \big)_{1\leqslant k\leqslant 3N}\in\mathbb R^{3N} ,
\end{equation*}
with for every $i,j,k \in\{1,\ldots,3N\}$, 
\begin{equation*}
(\Gamma (q) )^k_{i,j} := \frac12
\Big( \frac{\partial ({\mathcal M} (q))_{i,k}}{\partial q_{j}} 
+ \frac{\partial ({\mathcal M} (q))_{j,k}}{\partial q_{i}} 
 - \frac{\partial ({\mathcal M} (q))_{i,j}}{\partial q_{k}} \Big) .
\end{equation*}
\end{Definition}
Let us emphasize that the coefficients $ \Gamma^{k}_{i,j}$ are the Christoffel symbols of the first kind associated with the metric ${\mathcal M} (q)$.
\ \par 
The main result of this paper is the following reformulation of the dynamics of the rigid bodies as a second-order differential equation for $q$ with respect to the fluid vorticity and the circulations around the bodies, but without 
any reference to the fluid velocity nor the pressure. 
\begin{theorem} \label{THEO-intro}
 There exist a $C^{\infty}$ mapping $ q\in \mathcal Q \mapsto \mathcal{M}(q) \in S^{++}_{3N} (\mathbb{R})$, which depends only on the shape of $\mathcal{F}(q)$, on the masses $(m_{\kappa})_{1,\dots,N}$ and on the moments of inertia $(\mathcal{J}_{\kappa})_{1,\dots,N}$, 
 and a $C^{\infty}$ mapping on $\mathcal Q $ which associates with $q$ in $\mathcal Q $ the mapping
 $F(q,\cdot)$ from $\R^{3N}\times C^\infty ( \overline{\mathcal{F}(q)}) \times \R^N$ to $\R^{3N}$ such that $F(q,\cdot ,0,0)=0$, 
which depends only on the shape of $\mathcal{F}(q)$, such that
if $u$ and $q$ are smooth functions which satisfy System~\eqref{SYS_full_system},  
then, up to the first collision, $q$ satisfies the second order ODE:
\begin{equation} 
	\label{solid-normal}
\mathcal{M}(q) q^{ \prime\prime} + \langle \Gamma (q),q^{\prime},q^{\prime}\rangle = F (q,q^{ \prime},\omega,\gamma) ,
\end{equation}
where $\Gamma$ is the connection associated with $q \mapsto \mathcal{M}(q)$.
\end{theorem}
In the case where the vorticity and the circulations vanish, then the right hand side of \eqref{solid-normal} also vanishes and 
the equation then reduces to the geodesic equation associated with the metric ${\mathcal M} (q) $. This case was already treated in \cite{Munnier:2008ab}. 
The matrix $\mathcal{M}(q)$ is associated with the added mass effect. A well-known formula with respect to the so-called Kirchhoff potentials is recalled in Section \ref{Subsec:inertiamatrices}.  
On the other hand some computations of the Christoffel symbols are given in \cite[Lemma 6.3]{Munnier:2008ab}. Some alternative formulas are given in Section~\ref{Subsec:prop}. 

When the vorticity and the circulations are nonzero the force $F$ in the right hand side of \eqref{solid-normal} has to be taken into account. 
In particular the determination of the motion of the bodies has then to be completed with 
 the evolution equation for  the fluid vorticity  $ \omega$ which is 
 \begin{equation}
  \label{transp-vort}
 \frac{\partial \omega }{\partial t}+(u\cdot\nabla) \omega =0 , \quad x \in \mathcal{F}(q(t)) .
\end{equation}

In Section~\ref{Subsec:explicit} we give a more explicit expression of this force $F$, see \eqref{def-F}, which displays 
 gyroscopic features generalizing the Kutta-Joukowski-type lift force and reads as a non-linear coupling between body-body interactions through the potential flows that their velocity induces, 
body-body interactions through the irrotational flows that their circulation induces, and body-vorticity interactions. 

This reformulation extends the works \cite{GLS} and \cite{GLS2} where the case of one rigid body in an unbounded rotational flow was addressed
and the work \cite{GMS} where the case of one rigid body in a bounded irrotational flow was tackled. 
 \par 
 Let us mention two motivations for the result above. The first one is the 
 zero-radius limit as considered in \cite{GLS}, \cite{GLS2}, \cite{GMS} in more simple settings.
Another one is related to the controllability of the motion of the rigid bodies from a part of the external boundary, see 
 \cite{GKS} for the case of a single rigid body. 
 \par
 Regarding  the issue of collisions  we refer to \cite{H}, \cite{HM} and the recent paper \cite{C}.
 
 \bigskip

 The remainder of this article is divided into four Sections. We first determine the fluid velocity $u$ in terms of $(q,q',\omega, \gamma)$ in Section~\ref{Subsec:decomp}.
In Section~\ref{Subsec:explicit}, we give the explicit formulas of $\mathcal{M}, \Gamma , F$ which appear in \eqref{solid-normal}. Section~\ref{sec:normal} is dedicated to the proof of Theorem~\ref{THEO-intro} whereas Section~\ref{Subsec:prop} contains the proof of an explicit formula concerning $\Gamma$. 
Finally Section \ref{Subsec:DC} is devoted to the proof of Proposition~\ref{lem-DC} which highlights a property of a part of the force $F$ due to the circulations around the body.
%
%
%
%
%
\section{Decomposition of the fluid velocity into elementary velocities}
\label{Subsec:decomp}
In this section, we determine the fluid velocity $u$ in terms of $(q,q',\omega, \gamma)$.
This material follows from well-known property of incompressible flows in multiply-connected domains. 
We refer in particular to Kato \cite{Kato} for more details. 
\ \par
\ \par
\noindent
{\it Kirchhoff potentials.}
Consider for any $ \kappa \in \{1,2, \ldots , N \}$, 
the functions $\xi_{ \kappa,j} (q,\cdot)=\xi_{k} (q,\cdot)$, for $j=1,2,3$ and $k=3(\kappa-1)+j$, defined by $\xi_{ \kappa,j} (q,x) :=0 $ on $\partial\mathcal{F}(q)\setminus \partial \mathcal{S}_{\kappa}$ and by $\xi_{ \kappa,j} (q,x) := e_{j}, \text{ for } j=1,2$, and $ \xi_{ \kappa,3} (q,x) := (x-h_\kappa)^\perp \text{ on }\partial\mathcal{S}_{\kappa}.$
Above $e_1$ and $e_2$ are the unit vectors of the canonical basis.
We denote by $K_{ \kappa,j} (q,\cdot)=K_{k} (q,\cdot)$ the normal trace of $\xi_{ \kappa,j} $ on $ \partial \mathcal{F}(q)$, that is:
$K_{ \kappa,j} (q,\cdot) := n \cdot \xi_{ \kappa,j} (q,\cdot) \text{ on } \partial\mathcal{F}(q)$,
where $n$ denotes the unit normal vector pointing outside ${\mathcal F}(q)$. We introduce the Kirchhoff potentials $\varphi_{ \kappa,j}(q,\cdot)=\varphi_{k}(q,\cdot)$, for $j=1,2,3$ and $k=3(\kappa-1)+j$, which are the unique (up to an additive constant) solutions in $\mathcal F(q)$ of the following Neumann problem:
\begin{subequations} \label{Kir}
\begin{alignat}{3} \label{Kir1}
\Delta \varphi_{ \kappa,j} &= 0 & \quad & \text{ in } \mathcal F(q),\\ \label{Kir2}
\frac{\partial \varphi_{ \kappa,j}}{\partial n} (q,\cdot)&= K_{\kappa,j} (q,\cdot) & \quad & \text{ on }\partial\mathcal{F}(q).
\end{alignat}
\end{subequations}
We also denote 
\begin{equation} \label{Kir-gras}
\boldsymbol{K}_\kappa(q,\cdot) :=(K_{\kappa,1}(q,\cdot),K_{\kappa,2}(q,\cdot),K_{\kappa,3}(q,\cdot))^{t} \text{ and }
\boldsymbol\varphi_\kappa(q,\cdot) :=(\varphi_{\kappa,1}(q,\cdot),\varphi_{\kappa,2}(q,\cdot),\varphi_{\kappa,3}(q,\cdot))^{t}.
\end{equation}
Following our rules of notation for $q$, we define as well the function $\varphi(q,\cdot)$ by concatenating 
into a vector of length $3N$ the functions $\boldsymbol\varphi_\kappa(q,\cdot)$, namely:
$\varphi(q,\cdot) :=(\boldsymbol\varphi_1(q,\cdot)^t,\ldots,\boldsymbol\varphi_{N}(q,\cdot)^t)^t.$
\ \par
\ \par
\noindent
{\it Stream functions for the circulation.} To take into account the circulations of the velocity around the solids, we introduce for each $\kappa \in \{ 1, \dots, N\}$ the stream function $\psi_{\kappa}= \psi_{\kappa}(q,\cdot)$ defined on $\mathcal F(q)$ of the harmonic vector field which has circulation $\delta_{\kappa,\nu}$ around $\partial \mathcal{S}_{\nu}(q)$. More precisely, for every $q$, there exists a unique family $(C_{\kappa,\nu}(q))_{\nu \in \{1,2, \ldots , N \}}\in\mathbb R^N$ such that the unique solution 
$\psi_{\kappa}(q,\cdot)$ of the Dirichlet problem:
\begin{subequations} 
\label{def_stream}
\begin{alignat}{3}
\Delta \psi_{\kappa}(q,\cdot) & =0 & \quad & \text{ in } \mathcal F(q) \\
\psi_{\kappa}(q,\cdot) & = C_{\kappa,\nu}(q) & \quad & \text{ on } \partial \mathcal S_{\nu}(q), \text{ for } \nu \in \{1,2, \ldots , N \} , \\
\psi_{\kappa}(q,\cdot) & =0 & & \text{ on } \partial\Omega,
\end{alignat}
satisfies
\begin{equation} \label{circ-norma}
\int_{\partial\mathcal S_{\nu}(q)} \frac{\partial \psi_{\kappa}}{\partial n} (q,\cdot) {\rm d}s=-\delta_{\kappa,\nu}, \text{ for } \nu \in \{1,2, \ldots , N \},
\end{equation}
\end{subequations}
where $ \delta_{\nu, \kappa}$ denotes the Kronecker symbol.
As before, we define the concatenation into a vector of length $N$:
$\psi(q,\cdot):=(\psi_1(q,\cdot),\ldots,\psi_N(q,\cdot))^t$.
\ \par
\ \par
\noindent
{\it Hydrodynamic stream function.} For every smooth scalar function $\omega$ over $\mathcal{F}(q)$, there exists a unique family $(C_{\omega,\nu}(q))_{\nu \in \{1,2, \ldots , N \}}\in\mathbb R^N$ such that the unique solution $\psi_{\omega}(q,\cdot)\in H^1(\mathcal{F}(q))$ of:
\begin{subequations} 
\label{def_hydro-stream}
\begin{alignat}{3}
\Delta \psi_{\omega}(q,\cdot) & =\omega & \quad & \text{ in } \mathcal F(q) \\
\psi_{\omega}(q,\cdot) & = C_{\omega,\nu}(q) & \quad & \text{ on } \partial \mathcal S_{\nu}(q), \text{ for } \nu \in \{1,2, \ldots , N \} , \\
\psi_{\omega}(q,\cdot) & =0 & & \text{ on } \partial\Omega,
\end{alignat}
satisfies
\begin{equation} \label{circ-hydro}
\int_{\partial\mathcal S_{\nu}(q)} \frac{\partial \psi_{\omega}}{\partial n} (q,\cdot) {\rm d}s=0, \text{ for } \nu \in \{1,2, \ldots , N \}.
\end{equation}
\end{subequations}
Then the decomposition of the fluid velocity $u$ is given by the following result. 
\begin{prop}
Let $q\in \mathcal{Q}$. For any $q'$ in $\R^{3N}$, for any 
$\omega$ smooth over $\mathcal{F}(q)$ and for any $\gamma$ in $\R^{N}$,
 there exists a unique smooth vector field $u$ such that 
\begin{gather*}
\div u =0 \text{ in } \mathcal{F} (q), \quad
\curl u =\omega \text{ in } \mathcal{F} (q), \quad u\cdot n =0 \text{ on } \partial\Omega, \\
u \cdot n = \big( \theta^{\prime}_\kappa (\cdot-h_\kappa)^\perp + h^{ \prime}_\kappa \big) \cdot n 
\text{ on }\partial\mathcal S_\kappa(t) 
\text{ and } \int_{\partial\mathcal S_\kappa (t)} u(t)\cdot\tau \, {\rm d}s = \gamma_\kappa , \quad  \text{ for all } \ \kappa  \in \{1,2, \ldots , N \}.
\end{gather*}
\par
Moreover, we have the decomposition
\begin{equation} \label{EQ_irrotational_flow}
u = u_1 (q,q',\cdot) + u_2 (q,\omega,\gamma,\cdot) , 
\end{equation}
with 
\begin{gather} \label{DecompU1}
u_1 (q,q',\cdot) := \sum_{\kappa=1}^N \nabla (\boldsymbol\varphi_{\kappa} (q,\cdot)\cdot \boldsymbol q'_{\kappa} )
= \sum_{\kappa=1}^N \nabla \left( \sum_{j=1}^3 \varphi_{\kappa,j} (q,\cdot)q'_{\kappa,j} \right)= \sum_{k=1}^{3N} q'_{k}\nabla \varphi_{k},
\\ \label{def-psi}
 u_2 (q,\omega,\gamma,\cdot) := \nabla^\perp\psi_{\omega,\gamma} (q,\cdot) , \text{ where }
\psi_{\omega,\gamma}(q,\cdot) :=\psi_{\omega}(q,\cdot) + \psi(q,\cdot)\cdot\gamma.
\end{gather}
\end{prop}

Let us emphasize that when the 
fluid is irrotational the motion of the fluid is 
 completely determined by the motion of the bodies. 
 Otherwise the equation \eqref{transp-vort} has to be taken into account to determine the evolution of the fluid vorticity.

\section{Explicit formulas of the coefficients in the reformulation of the dynamics of the rigid bodies}
\label{Subsec:explicit}
 In this section we give some explicit expressions of the quantities $\mathcal{M}(q)$ and $F(q)$ which are involved in \eqref{solid-normal}. 
 These expressions give more insights on the various effects involved in the dynamics of the rigid bodies. 
 \subsection{Inertia matrices}
 \label{Subsec:inertiamatrices}
Let us gather the mass and moment of inertia of the solid $\kappa$ into the following matrix:
\begin{equation*} \label{DefMGkappa}
{\mathcal M}^g_{\kappa} := \begin{pmatrix}
	m_\kappa & 0 & 0 \\
	0 & m_\kappa & 0 \\
	0 & 0 & {\mathcal J}_\kappa
\end{pmatrix} .
\end{equation*}
The total genuine inertia matrix of the system is the $3N\times 3N$ bloc diagonal matrix:
\begin{equation}
\label{DefMG}
 \mathcal{M}^g:={\rm diag}( \mathcal{M}^g_1,\ldots, \mathcal{M}^g_N).
\end{equation}
which is diagonal and in the set $S^{++}_{3N} (\mathbb{R})$ of the real symmetric positive definite $3N\times 3N$ matrices.
In addition to their genuine inertia the rigid bodies have also to accelerate the surrounding fluid as they move through it. 
This is the phenomenon of added inertia, which is encoded by the matrices: 
\begin{equation}
 \mathcal{M}^{a}_{\kappa,\nu}(q):=\int_{\partial\mathcal S_\kappa(q)}\boldsymbol\varphi_\nu(q,\cdot)\otimes \frac{\partial\boldsymbol\varphi_\kappa}{\partial n}(q,\cdot)\,\, {\rm d}s.
\end{equation}
Integrating by parts, we also have
\begin{equation}
 \mathcal{M}^{a}_{\kappa,\nu}(q) =\int_{\mathcal  F (q)} \nabla \boldsymbol\varphi_\nu(q,\cdot)\otimes \nabla \boldsymbol\varphi_\kappa (q,\cdot)\,\, {\rm d}s=
\int_{\partial\mathcal S_\nu(q)}\boldsymbol\varphi_\kappa(q,\cdot)\otimes \frac{\partial\boldsymbol\varphi_\nu}{\partial n}(q,\cdot)\, {\rm d}s.
\end{equation}
The total added inertia matrix of the system reads: 
\begin{equation}
\label{DefMGa}
 \mathcal{M}^{a}(q) := \left( \mathcal{M}^{a}_{\kappa,\nu}(q)\right)_{1\leqslant \kappa, \nu\leqslant N}=\int_{\partial\mathcal F(q)}\varphi(q,\cdot)\otimes \frac{\partial\varphi}{\partial n}(q,\cdot)\, {\rm d}x.
\end{equation}
Notice that this matrix is in the set $S^{+}_{3N} (\mathbb{R})$ of the real symmetric positive-semidefinite $3N\times 3N$ matrices. These matrices depend on the shape of the fluid domain, therefore on the shape of the external boundaries and of the shape and position of the rigid bodies. 
\par
Finally, the overall inertia of the system is encoded in the 
inertia matrix:
\begin{equation}\label{def-M}
\mathcal{M}(q):=\mathcal{M}^g+\mathcal{M}^a(q) \in S^{++}_{3N} (\mathbb{R}).
\end{equation}
%
%
%
\subsection{Force term}
 \label{sec-marre}
From the stream function $\psi_{\omega,\gamma}$ associated with $\omega,\gamma$ \eqref{def-psi}, and the Kirchhoff potentials \eqref{Kir}-\eqref{Kir-gras}, we define:
\begin{align*} 
 \mathcal{A}_{\kappa,\nu}(q,\omega,\gamma):=& 
 \int_{\partial\mathcal S_{\kappa}(q)} \frac{\partial \psi_{\omega,\gamma}}{\partial n}(q,\cdot) \left( \frac{\partial\boldsymbol\varphi_{\kappa}}{\partial n} \otimes \frac{\partial\boldsymbol\varphi_{\nu}}{\partial \tau} \right) (q,\cdot) \, {\rm d}s\\
& - \int_{\partial\mathcal S_{\nu}(q)} \frac{\partial \psi_{\omega,\gamma}}{\partial n}(q,\cdot) \left( \frac{\partial\boldsymbol\varphi_{\kappa}}{\partial \tau} \otimes \frac{\partial\boldsymbol\varphi_{\nu}}{\partial n} \right) (q,\cdot) \, {\rm d}s , \\
\mathcal{A}(q,\omega,\gamma) :=& \Big( \mathcal{A}_{\kappa,\nu}(q,\omega,\gamma) \Big)_{1\leq \kappa,\nu\leq N} .
\end{align*}
We observe that the matrix $\mathcal{A}$ is skew-symmetric and linear with respect to $\omega$ and $\gamma$.
With this matrix we associate the gyroscopic force $\mathcal{A}(q,\omega,\gamma) q'$ which generalizes the celebrated Kutta-Joukowski force, see for instance \cite{SueurP}. 
\par
We also define:
\begin{gather*} 
E_{\kappa}(q,\omega,\gamma) := - \frac{1}{2} \int_{\partial\mathcal S_{\kappa} (q)} 
\left( \left| \frac{\partial\psi_{\omega,\gamma}}{\partial n} \right|^2 \frac{\partial\boldsymbol\varphi_{\kappa}}{\partial n} \right) (q,\cdot)\, {\rm d}s ,
\\ E(q,\omega,\gamma):=(E_{1}(q,\omega,\gamma) ,\ldots,E_{N}(q,\omega,\gamma))^t,
\\ D_k(q,q',\omega,\gamma):=-\int_{\mathcal F(q)}\omega u^\perp(q,q',\omega,\gamma,\cdot )\cdot\nabla \varphi_{k}(q,\cdot)\,{\rm d}x,\\
D(q,q',\omega,\gamma ):=(D_1(q,q',\omega,\gamma),\ldots,D_{3N}(q,q',\omega,\gamma))^t .
\end{gather*}
Finally we set 
\begin{equation}
 \label{def-F}
F (q,q',\omega,\gamma) := E(q,\omega,\gamma) + \mathcal{A}(q,\omega,\gamma) q' + D(q,q',\omega,\gamma ) .
\end{equation}
 It follows from the definitions and properties of the elementary flows given in the previous section that the mapping which associates with $q$ in $\mathcal Q $ the mapping
 $F(q,\cdot)$ from $\R^{3N}\times C^\infty (\overline{\mathcal{F}(q)}) \times \R^N$ to $\R^{3N}$ is $C^{\infty}$, such that $F(\cdot,\cdot ,0,0)=0$, 
and depends only on the shape of $\mathcal{F}(q)$.
\par
Moreover, without vorticity, we can compute $E$ in terms of the shape derivatives of the constants $C_{\kappa,\nu} (q)$ appearing in the definition of the stream functions for the circulation terms \eqref{def_stream}.
\begin{prop}\label{lem-DC}
For any $q$ in $\mathcal{Q}$ and for any $\gamma$ in $\R^{N}$, 
 \[
 E(q,0,\gamma)=\frac12 \nabla_q \Big(\sum_{m=1}^N\sum_{\nu=1}^N \gamma_{\kappa}\gamma_{\nu }C_{\kappa,\nu} (q) \Big) .
 \]
\end{prop}
%
%
\section{Proof of the main result}
\label{sec:normal}
This section is devoted to the proof of Theorem~\ref{THEO-intro}.
For each integer ${\kappa}$ between $1$ and $N$, $q\in \mathcal Q$, $\ell_{\kappa}^\ast\in \R^2$ and $r_{k}^\ast \in \R$,  we consider the following potential vector field $\mathcal F(q)$: 
$u^\ast_{\kappa} :=\nabla (\boldsymbol\varphi_{\kappa} (q,\cdot)\cdot p^\ast_{\kappa})$ where $p^\ast_{\kappa}:=(\ell^{\ast t}_{\kappa},r^\ast_{\kappa})^t$.
By \eqref{Kir1}, 
\begin{equation} \label{bc*}
u^\ast_{\kappa} \cdot n = \delta_{\kappa ,\nu} (\ell^\ast_{\kappa} + r^\ast_{\kappa}(\cdot - h_{\kappa})^\perp) \cdot n \text{ on } \partial\mathcal S_{\nu} (q) \text{ and } u^\ast_{\kappa} \cdot n =0 \text{ on } \partial \Omega , 
\end{equation}
Combining with \eqref{EqTrans} and \eqref{EqRot}, we arrive at
\begin{equation}
\label{bern_0}
m_{\kappa} h_\kappa'' \cdot\ell^\ast_{\kappa}+\mathcal J_{\kappa} \theta_\kappa'' r^\ast_{\kappa} =
\int_{\partial \mathcal F(q)} \pi u^\ast_{\kappa} \cdot n \, {\rm d}s .
\end{equation}
Since, by \eqref{Kir2}, $u^\ast_{\kappa}$ is divergence-free in $\mathcal F(q)$, with an integration by parts we arrive at 
\begin{equation*}
m_{\kappa} h_\kappa'' \cdot\ell^\ast_{\kappa}+\mathcal J_{\kappa} \theta_\kappa'' r^\ast_{\kappa} =
\int_{\mathcal F(q)} \nabla \pi \cdot u^\ast_{\kappa} \, {\rm d}x .
\end{equation*}
Since, by \eqref{EEE1}, 
\begin{equation*}
\nabla \pi =-\left(\frac{\partial u}{\partial t}+\frac{1}{2}\nabla|u |^2 + \omega u^{\perp} \right) \quad \text{in }\mathcal F(q) ,
\end{equation*}
we deduce from \eqref{bern_0} that
\begin{equation*}
m_{\kappa} h_\kappa'' \cdot\ell^\ast_{\kappa}+\mathcal J_{\kappa} \theta_\kappa'' r^\ast_{\kappa}=
-\int_{\mathcal F(q)}\left(\frac{\partial u}{\partial t}+
\frac{1}{2}\nabla|u|^2 + \omega u^{\perp}
\right)\cdot u^\ast_{\kappa} \, {\rm d}x .
\end{equation*}
\par
Summing over all the ${\kappa}$ and using the decomposition \eqref{EQ_irrotational_flow}
therefore yields 
\begin{align*}
{\mathcal M}^g q^{ \prime\prime} \cdot p^\ast 
+ \int_{\mathcal F(q)} \Big( \frac{\partial u_1}{\partial t}+\frac{1}{2}\nabla|u_1|^2 \Big) \cdot u^\ast\, {\rm d}x
& = - \int_{\mathcal F(q)} \big( \frac{1}{2}\nabla|u_2|^2 \big) \cdot u^\ast\, {\rm d}x \\
& \quad - \int_{\mathcal F(q)}\big(\frac{\partial u_2}{\partial t} + \nabla (u_{1} \cdot u_{2} ) \big) \cdot u^\ast\, {\rm d}x \\
& \quad -\int_{\mathcal F(q)} \omega u^{\perp}\cdot u^\ast \, {\rm d}x ,
\end{align*}
for all $p^\ast \in\mathbb R^{3N}$, 
where
\begin{equation}
\label{defuast}
 u^\ast := \sum_{1 \leqslant \kappa \leqslant N} u^\ast_{\kappa}.
 \end{equation}
We also recall that ${\mathcal M}^g$ is defined in \eqref{DefMG}.
The lemma below is proved in \cite[Lemma 5.1]{Munnier:2008ab}.
\begin{lemma} \label{LEM_3}
For any smooth curve $q(t)$ in $\mathcal Q$ and every $p^\ast \in\mathbb R^{3N}$, the following identity holds:
\begin{equation*}
{\mathcal M}^g q^{ \prime\prime} \cdot p^\ast 
+ \int_{\mathcal F(q)} \left(\frac{\partial u_1}{\partial t}+\frac{1}{2}\nabla|u_1|^2\right)\cdot u^\ast\, {\rm d}x
= {\mathcal M}(q) q^{ \prime\prime} \cdot p^\ast 
+
\langle \Gamma(q),q',q'\rangle \cdot p^\ast ,
\end{equation*}
where ${\mathcal M}$ is given in \eqref{def-M} and 
$\Gamma (q) $ is the connection associated with $ q \mapsto \mathcal{M}(q)$.
\end{lemma}
\ \par 
Before  moving on we recall that the definitions of $E(q,\omega)$ and $\mathcal{A}(q,\omega)$ are given in the previous section. 
As $|u_2| = |u_{2}\cdot \tau |$ on the boundary, by an integration by parts, we have the following.
\begin{lemma} \label{LEM_easy}
For every $q\in\mathcal Q$ and every $p^\ast \in\mathbb R^{3N}$, the following identity holds:
\begin{equation*}
- \int_{\mathcal F(q)}\left(\frac{1}{2}\nabla|u_2|^2\right)\cdot u^\ast\, {\rm d}x= E(q,\omega)\cdot p^\ast ,
\end{equation*}
where $u^\ast$ is given by \eqref{defuast}.
\end{lemma}
On the other hand, we have the following result regarding the crossed term. 
\begin{lemma} \label{lem-F-A}
For any smooth curve $q(t)$ in $\mathcal Q$ and every $p^\ast \in\mathbb R^{3N}$, the following identity holds:
\begin{equation*}
-\int_{\mathcal F(q)}\big(\frac{\partial u_2}{\partial t} + \nabla (u_{1} \cdot u_{2} ) \big) \cdot u^\ast\, {\rm d}x= \left(\mathcal{A}(q,\omega,\gamma) q'\right)\cdot p^\ast ,
\end{equation*}
where $u^\ast$ is given by \eqref{defuast}.
\end{lemma}
\begin{proof}
By an integration by parts,
\begin{equation}
 \label{BR3}
- \int_{\mathcal F(q)} \frac{\partial u_2}{\partial t} \cdot u^\ast\, {\rm d}x
= \sum_{\nu}\int_{\partial\mathcal S_{\nu}(q)} \frac{\partial }{\partial t} \left(\psi_{\omega,\gamma}(q,\cdot) \right) (u^\ast \cdot \tau)\, {\rm d}s .
\end{equation}
Now we use the following lemma where we set $C_{\omega,\gamma,\nu} := C_{\omega,\nu} + \sum_{\kappa=1}^N \gamma_{\kappa} C_{\kappa,\nu}$. 
%
%
\begin{lemma}\label{lem37-GMS}
For any smooth curve $q(t)$ in $\mathcal Q$,
 for any $\nu$,
on $\partial \mathcal{S}_{\nu}(q)$:
\begin{equation*}
 \frac{\partial }{\partial t} \left( \psi_{\omega,\gamma}(q,\cdot) \right) =- \frac{\partial \psi_{\omega,\gamma}}{\partial n}(q,\cdot) \left( \frac{\partial \boldsymbol\varphi_{\nu} }{\partial n}(q,\cdot) \cdot \boldsymbol q'_{\nu}\right)
 +\frac{\partial C_{\omega,\gamma,\nu}}{\partial q}(q)\cdot q'  
 +\frac{\delta C_{\omega,\nu}}{\delta \omega}(q) \cdot \frac{\partial\omega}{\partial t} .
\end{equation*}
\end{lemma}
Let us stress that $ \frac{\delta C_{\omega,\nu}}{\delta \omega}(q) \cdot \frac{\partial\omega}{\partial t}$ denotes  the 
 functional derivative of $C_{\omega,\nu} $ with respect to  $\omega$, so that the last term is constant on $\partial \mathcal{S}_{\nu}(q)$ as a function of $x$.
\begin{proof}[Proof of Lemma~\ref{lem37-GMS}]
 By Kelvin's theorem, we consider only solutions such that $\gamma'(t)=0$, and we start with the observation that 
\begin{equation}
\label{samsoul1}
 \frac{\partial}{\partial t} \left( \psi_{\omega,\gamma}(q,\cdot) \right) = \frac{\partial\psi_{\omega,\gamma} }{\partial q} \cdot q'      + 
 \frac{\delta\psi_{\omega,\gamma}}{\delta \omega}  (q,\cdot)  \cdot \frac{\partial\omega}{\partial t}      .
 \end{equation}
\par
By differentiating with respect to $t$ the identity:
$\psi_{\omega,\gamma} (q,R(\theta_{\nu})(X-h_{\nu,0})+h_{\nu})= C_{\omega,\gamma,\nu}(q)$, for $ X\in\partial\mathcal S_{\nu,0},$
and by setting $x = R(\theta_{\nu})(X-h_{\nu,0})+h_{\nu}$, 
we obtain for every $x\in\partial\mathcal S_{\nu}(q)$, 
\begin{equation}
\label{samsoul2}
\frac{\partial\psi_{\omega,\gamma}}{\partial q}(q,x)\cdot q' + \frac{\delta \psi_{\omega,\gamma}}{\delta \omega} (q,x) \cdot    \frac{\partial\omega}{\partial t} 
+ \nabla \psi_{\omega,\gamma} (q,x)\cdot w_{\nu} (q,x) = \frac{\partial C_{\omega,\gamma,\nu}}{\partial q}(q)\cdot q'    + 
\frac{\delta C_{\omega,\nu}}{\delta \omega}(q) \cdot \frac{\partial\omega}{\partial t} ,
\end{equation}
where $w_{\nu}(q, x)= \theta'_{\nu} (x-h_{\nu})^\perp+h'_{\nu}$.
Since $\psi_{\omega,\gamma} (q, \cdot)$ is constant on $\partial\mathcal S_{\nu}(q)$, its tangential derivative is zero. 
Besides, on $\partial\mathcal S_{\nu}(q)$ we have $w_{\nu} (q,\cdot) \cdot n= \frac{\partial\boldsymbol\varphi_{\nu}}{\partial n}(q,\cdot) \cdot \boldsymbol q'_{\nu}$.
Thus we get
\begin{equation}
\label{samsoul3}
\nabla \psi_{\omega,\gamma} (q,x)\cdot w_{\nu} (q,x)
= \frac{\partial\psi_{\omega,\gamma}}{\partial n}(q,x) \big( \frac{\partial\boldsymbol\varphi_{\nu}}{\partial n}(q,x)\cdot \boldsymbol q'_{\nu} \big) 
\end{equation}
for $ x\in\partial\mathcal S_{\nu}(q)$.
By combining \eqref{samsoul1}, \eqref{samsoul2} and \eqref{samsoul3}
 we conclude the proof of Lemma~\ref{lem37-GMS}.
\end{proof}
By \eqref{BR3} and Lemma~\ref{lem37-GMS}, we get
\begin{align}\label{dt-u2-ustar}
 - \int_{\mathcal F(q)} \frac{\partial u_2}{\partial t} \cdot u^\ast\, {\rm d}x
& = -\sum_{\kappa=1}^N \sum_{\nu=1}^N \left(\int_{\partial \mathcal S_{\nu}(q)} \frac{\partial \psi_{\omega,\gamma} }{\partial n} 
\left(\frac{\partial \boldsymbol\varphi_{\kappa} }{\partial \tau} \otimes \frac{\partial \boldsymbol\varphi_{\nu} }{\partial n}\right){\rm d}s\ \boldsymbol q'_{\nu} \right) 
\cdot p^{\ast}_{\kappa} .
\end{align}
Let us stress here that the last  two terms in the equality of Lemma~\ref{lem37-GMS} are constant so that their contribution in \eqref{BR3} vanish since $\int_{\partial \mathcal S_{\nu}(q)}  u^\ast \cdot \tau \, {\rm d}s=0$. 
\par
On the other hand, since $u^{\ast}$ is divergence-free, we have the Stokes formula:
\begin{equation*}
\int_{\mathcal F(q)} \nabla (u_{1} \cdot u_{2} ) \cdot u^\ast\, {\rm d}x = \int_{\partial \mathcal F (q)} (u_{1} \cdot u_{2} ) (u^\ast\cdot n)\, {\rm d}s .
\end{equation*}
By \eqref{DecompU1}-\eqref{def-psi}, we have on $\partial \mathcal S_{\kappa}(q)$, 
\[
u_{1} \cdot u_{2} = (u_{1}\cdot \tau)(u_{2}\cdot \tau)
=-\sum_{\nu=1}^N \frac{\partial \boldsymbol\varphi_{\nu} }{\partial \tau}\cdot \boldsymbol q'_{\nu} \frac{\partial \psi_{\omega,\gamma} }{\partial n} ,
\]
hence we obtain by the definition of $u^\ast$  that 
\begin{align}
-\int_{\mathcal F(q)} \nabla (u_{1} \cdot u_{2} ) \cdot u^\ast\, {\rm d}x
\label{nabla-u1u2-ustar}
&=\sum_{\kappa=1}^N \sum_{\nu=1}^N \left(\int_{\partial \mathcal S_{\kappa}(q)} \frac{\partial \psi_{\omega,\gamma} }{\partial n}\left(\frac{\partial \boldsymbol\varphi_{\kappa} }{\partial n} \otimes \frac{\partial \boldsymbol\varphi_{\nu} }{\partial \tau}\right)\, {\rm d}s \, \boldsymbol q'_{\nu} \right) \cdot p^{\ast}_{\kappa}.
\end{align}
\ \par
Gathering \eqref{dt-u2-ustar} and \eqref{nabla-u1u2-ustar}, we end the proof of Lemma~\ref{lem-F-A}.
\end{proof}
\ \par 
Combining Lemmas~\ref{LEM_3}, \ref{LEM_easy}, \ref{lem-F-A}, observing that 
\begin{equation*}
-\int_{\mathcal F(q)} \omega u^{\perp}\cdot u^\ast \, {\rm d}x = 
 D (q,q',\omega ) \cdot p^\ast
\end{equation*}
and taking \eqref{def-F} into account, 
 we arrive at \eqref{solid-normal} (since $p^\ast$ can be arbitrarily chosen in $\mathbb R^{3N}$).
This concludes the proof of Theorem~\ref{THEO-intro}. 
%
%
%
\section{Computation of the Christoffel symbols}
\label{Subsec:prop}
 %
This section is devoted to the proof of the following result regarding the computations of the Christoffel symbols.
\begin{prop} \label{Prop-L}
For every integers $i,j,k$ between $1$ and $3N$, 
\begin{align*} 
 \frac{\partial {\mathcal M}^{a}_{i,j}(q)}{\partial q_{k}} 
=& -\int_{\partial\mathcal S_{[k]}(q)} \Big[ (\xi_i - \nabla \varphi_i) \cdot (\xi_j - \nabla \varphi_j) -\xi_{i} \cdot \xi_{j} \Big]\, (\xi_k \cdot n) \, {\rm d}s. \\
\nonumber
&-\delta_{3,(i)}\delta_{[k],[i]}\delta_{\{1,2\},(k)}\int_{\partial\mathcal S_{[k]}( q)}\varphi_j (\xi_k\cdot \tau){\rm d}s-\delta_{3,(j)}\delta_{[k],[j]}\delta_{\{1,2\},(k)} \int_{\partial\mathcal S_{[k]}(q)}\varphi_i (\xi_k \cdot \tau){\rm d}s.
\end{align*}
\end{prop}
Observe that the formula above does not involve the curvature of the fluid domain as opposed  to \cite[Lemma 6.3]{Munnier:2008ab}. 
\begin{proof}
Let $\kappa\in\{1,\ldots,N\}$, $p_\kappa^\ast=(\ell^{\ast t}_\kappa,r_\kappa^\ast)^t\in\mathbb R^3$, $q\in\mathcal Q$ and $\xi_\kappa^\ast$ be a smooth divergence-free vector field such that $\xi_\kappa^\ast(x)=r_\kappa^\ast(x-h_{\kappa})^\perp+\ell_\kappa^\ast$ in a neighborhood of $\mathcal S_{\kappa}(q)$ and $\xi_\kappa^\ast(x)=0$ elsewhere.
Define the diffeomorphism $T_\kappa(t,\cdot)$ (for $t$ small) as being the flow, corresponding to the Cauchy problem:
$$\frac{\partial T_\kappa}{\partial t}(t,x)=\xi^\ast_\kappa(T_\kappa(t,x)),\qquad T_\kappa(0,x)=x.$$
Notice that we have:
$$\frac{\partial T_\kappa}{\partial t}(0,x)=\xi_\kappa^\ast(x).$$
Define now the  sets $\mathcal F(t):=T_\kappa(t,\mathcal F(q))$, $\mathcal S_{\mu}(t):=T_\kappa(t,\mathcal S_{\mu}(q))$ for all $\mu=1,\ldots,N$ (noting that $\mathcal S_{\mu}$ is fixed for $\mu\neq \kappa$) and, for every $t$ small and every $i\in\{1,\ldots,3N\}$, and the Kirchhoff potentials $t\mapsto (\overline{\varphi}_i(t,\cdot))_{i=1,\dots,3N}$ on $\mathcal F(t)$, see \eqref{Kir}.

The shape derivative of the Kirchhoff potential $\varphi_i$ ($i=1,\ldots,3N$) at the point $q$ in the direction $p_\kappa^\ast=(\ell_\kappa^{\ast t},r_\kappa^\ast)^t$ ($\kappa=1,\ldots,N$), is then defined by
$$\left(\frac{\partial\varphi_i}{\partial \boldsymbol q_\kappa}\cdot p_\kappa^\ast\right)(q,\cdot):=\overline{\varphi}_i'(0,\cdot) .$$
\begin{lemma}
 \label{lem-banane}
For  $q\in\mathcal Q$ and $p_\kappa^\ast=(\ell^{\ast t}_\kappa,r_\kappa^\ast)^t\in\mathbb R^3$, 
$\left(\frac{\partial\varphi_i}{\partial \boldsymbol q_\kappa}\cdot p_\kappa^\ast\right)(q,\cdot)$
is harmonic in $\mathcal F(q)$, satisfies the following Neumann condition on $\partial\mathcal  S_{\kappa}(q)$:
\begin{equation} \label{def_BC}
\frac{\partial}{\partial n}\left(\frac{\partial\varphi_i}{\partial \boldsymbol q_\kappa}\cdot p_\kappa^\ast\right)(q,\cdot)=
\begin{cases}
\frac{\partial}{\partial\tau}\left[\left(\frac{\partial \varphi_i}{\partial\tau}
-(\xi_i\cdot \tau)\right)(\xi_\kappa^\ast\cdot n)\right]
-\delta_{3,(i)} \, \ell_\kappa^\ast\cdot \tau&\text{if }\kappa=[i],  \\[0.2cm]
\frac{\partial}{\partial\tau}\left[\frac{\partial \varphi_i}{\partial\tau}
(\xi_\kappa^\ast\cdot n)\right]
&\text{if }\kappa\neq [i],
\end{cases}
\end{equation}
where $\xi_\kappa^\ast=r_\kappa^\ast(\cdot-h_\kappa)^\perp+\ell_\kappa^\ast$,
and the following homogeneous Neumann condition on $\partial\mathcal F(q) \setminus \partial\mathcal S_{\kappa}(q)$: 
\begin{equation} \label{def_BC2}
\frac{\partial}{\partial n}\left(\frac{\partial\varphi_i}{\partial \boldsymbol q_\kappa}\cdot p_\kappa^\ast\right)(q,\cdot)= 0.
\end{equation}
\end{lemma}
%
%
\begin{proof}
Differentiating with respect to $t$ at every point $x$ in $\mathcal F(q)$ the identity $-\Delta \overline{\varphi}_i(t,x)=0$ we get $-\Delta \overline{\varphi}'_i(0,x)=0$.
The boundary condition \eqref{Kir2} can be rewritten, after a change of variables, as:
\begin{equation} \label{eq:NC_1}
\nabla\overline{\varphi}_i(t,T_\kappa(t,x))\cdot n(t,T_\kappa(t,x))=\xi_i(t,T_\kappa(t,x))\cdot n(t,T_\kappa(t,x))\quad\text{ on }\partial \mathcal S_\kappa(q),
\end{equation}
where $n(t,\cdot)$ is the unit normal vector to $\partial\mathcal F(t)$ directed towards the exterior of $\mathcal F(t)$.
Notice now that $n(t,T_\kappa(t,\cdot))=R(tr_\kappa^\ast)n(0,\cdot)$ and hence that:
\begin{equation} \label{eq:NC_2}
\frac{\partial n}{\partial t}(0,\cdot)=-r_\kappa^\ast\tau(0,\cdot)\text{ on }\partial \mathcal S_\kappa(q).
\end{equation}
We deduce in particular that, when $[i]=\kappa$:
\begin{subequations} \label{eq:NC_3}
\begin{equation}
\frac{d}{dt}\left(\xi_i(t,T_\kappa(t,\cdot))\cdot n(t,T_\kappa(t,\cdot))\right)\big|_{t=0}=\begin{cases}-r_\kappa^\ast (e_{(i)}\cdot\tau) &\text{if }(i)=1,2,\\
0&\text{if }(i)=3,
\end{cases}
\end{equation}
because on $\mathcal{S}_{\kappa}$
\[
\frac{d}{dt}\Big((T_\kappa(t,x)-T_\kappa(t,h_{\kappa}))^\perp \cdot n(t,T_\kappa(t,\cdot)) \Big) \Big|_{t=0}
=0,
\]
and obviously, when $[i]\neq \kappa$:
\begin{equation}
\frac{d}{dt}\left(\xi_i(t,T_\kappa(t,\cdot))\cdot n(t,T_\kappa(t,\cdot))\right)\big|_{t=0}=0.
\end{equation}
\end{subequations}
Differentiating now \eqref{eq:NC_1} with respect to $t$ and using \eqref{eq:NC_2} and \eqref{eq:NC_3} we obtain:
\begin{equation} \label{eq:NC_main}
\frac{\partial \overline\varphi_i'}{\partial n}(0,\cdot)+D^2 \overline\varphi_i(0,\cdot) n\cdot \xi_\kappa^\ast-r_\kappa^\ast\frac{\partial \overline\varphi_i}{\partial\tau}(0,\cdot)=
\begin{cases}
0 & \text{if } \kappa \neq [i]\\
-r_\kappa^\ast (e_{(i)}\cdot\tau) & \text{if } \kappa=[i],\,(i)=1,2, \\
0 & \text{if } \kappa=[i],\,(i)=3.
\end{cases}
\end{equation}
where $D^2\overline\varphi_i(0,\cdot)$ is the Hessian matrix of $\overline\varphi_i(0,\cdot)$.
On the one hand, let us now decompose $\xi_\kappa^\ast$ into 
\begin{equation} \label{eq:NC_4}
\xi_\kappa^\ast=(\xi_\kappa^\ast\cdot n)n+(\xi_\kappa^\ast\cdot \tau)\tau\quad\text{ on }\partial\mathcal S_\kappa(q).
\end{equation}
On the other hand, computing the tangential derivative of the boundary condition \eqref{Kir2} on $\partial\mathcal S_\kappa(q)$, we get:
\begin{equation} \label{eq:NC_5}
D^2 \overline\varphi_i(0,\cdot)n\cdot\tau=\mathcal H\frac{\partial \overline\varphi_i}{\partial\tau}(0,\cdot)+
\begin{cases}
0 & \text{if } \kappa \neq[i] \\
-\mathcal H (e_{(i)}\cdot\tau) & \text{if } \kappa=[i], \, (i)=1,2, \\
1-\mathcal H(\cdot-h_{\kappa})^\perp \cdot \tau & \text{if } \kappa=[i],\,(i)=3,
\end{cases}
\end{equation}
where $\mathcal H$ is the local curvature of $\partial\mathcal F(q)$ (defined by $\frac{\partial n}{\partial\tau}=-\mathcal H\tau$ on $\partial\mathcal F(q)$).
On $\partial\mathcal F(q)$, we also have:
$$
D^2 \overline\varphi_i(0,\cdot)n\cdot n-\mathcal H\frac{\partial \overline\varphi_i}{\partial n}(0,\cdot)+
\frac{\partial}{\partial\tau}\left(\frac{\partial\overline\varphi_i}{\partial\tau}\right)(0,\cdot) = \Delta \overline\varphi_i(0,\cdot)=0,
$$
whence we deduce that:
\begin{equation} \label{eq:NC_6}
D^2 \overline\varphi_i(0,\cdot)n\cdot n = \mathcal H\frac{\partial \overline\varphi_i}{\partial n}(0,\cdot)
- \frac{\partial}{\partial\tau}\left(\frac{\partial\overline\varphi_i}{\partial\tau}\right)(0,\cdot).
\end{equation}
Plugging the decomposition \eqref{eq:NC_4} into \eqref{eq:NC_main} and using the equalities \eqref{Kir2}, \eqref{eq:NC_5} and \eqref{eq:NC_6}, we get:
\begin{multline*}
\frac{\partial \overline\varphi_i'}{\partial n}(0,\cdot)=\left[(\xi_\kappa^\ast\cdot n) \frac{\partial}{\partial\tau} \left( \frac{\partial\overline\varphi_i}{\partial\tau} \right)(0,\cdot)
+(r_\kappa^\ast-\mathcal H(\xi_\kappa^\ast\cdot\tau)) \frac{\partial \overline\varphi_i}{\partial \tau}(0,\cdot)\right] \\
+
\begin{cases}
0 & \text{if } \kappa \neq [i] \\
(-r_\kappa^\ast +\mathcal H (\xi_\kappa^\ast\cdot\tau))(e_{(i)}\cdot\tau)-\mathcal H(e_{(i)}\cdot n)(\xi_\kappa^\ast\cdot n) 
& \text{if } \kappa=[i],\, (i)=1,2, \\
(-r_\kappa^\ast +\mathcal H (\xi_\kappa^\ast\cdot\tau))(\cdot-h_{k})^\perp\cdot\tau -\mathcal H(\cdot-h_{k})^\perp\cdot n (\xi_\kappa^\ast\cdot n)
- \ell_\kappa^\ast\cdot\tau & \text{if } \kappa=[i], \, (i)=3.
\end{cases}
\end{multline*}
Observe now that:
$$\frac{\partial}{\partial\tau}(\xi_\kappa^\ast\cdot n)=r_\kappa^\ast-\mathcal H(\xi_\kappa^\ast\cdot\tau)\quad\text{ on }\partial\mathcal S_\kappa(q),$$
and the result follows.
\end{proof}
%
%
With the same notation as above, we can now compute the shape derivatives of the entries of the added mass matrix \eqref{DefMGa}. 
\begin{lemma}\label{express:christoff}
The following identity holds true for every $i,j \in \{1,\ldots,3N\}$, $\kappa \in \{1,\ldots,N\}$:
\begin{align*} 
\left(\frac{\partial \mathcal M^a_{ij}}{\partial \boldsymbol q_\kappa}\cdot p_\kappa^\ast\right)(q)
= &- \int_{\partial\mathcal S_{\kappa}(q)} \big( \nabla\varphi_i \cdot \nabla\varphi_j 
- \nabla\varphi_i \cdot \xi_{j} - \nabla\varphi_j \cdot \xi_{i} \big) \, (\xi_\kappa^\ast\cdot n) \,{\rm d}s \\
&- \delta_{3,(i)} \delta_{\kappa,[i]} \int_{\partial\mathcal S_\kappa(q)}\varphi_j (\ell_{\kappa}^\ast\cdot \tau){\rm d}s
- \delta_{3,(j)} \delta_{\kappa,[j]} \int_{\partial\mathcal S_\kappa(q)}\varphi_i (\ell_{\kappa}^\ast\cdot \tau){\rm d}s.
\end{align*}
where we recall that $p_\kappa^\ast=(\ell_\kappa^{\ast t},r_\kappa^\ast)^t$ and $\xi_\kappa^\ast=r_\kappa^\ast(\cdot-h_\kappa)^\perp+\ell_\kappa^\ast$.
\end{lemma}
%
\begin{proof}
The quantity we are interested in is the derivative at $t=0$ of:
$$\mathcal M^a_{ij}(t):=\int_{\mathcal F(t)}\nabla\overline{\varphi}_i\cdot \nabla\overline{\varphi}_j\, {\rm d}x.$$
The proof relies on Reynold's transport theorem which reads:
$$
\frac{d \mathcal M^a_{ij}}{dt}(0)
:= \int_{\mathcal F(q)}(\nabla\overline{\varphi}'_i \cdot \nabla\overline{\varphi}_j 
+ \nabla\overline{\varphi}_i\cdot \nabla\overline{\varphi}'_j)\, {\rm d}x 
+ \int_{\partial \mathcal S_\kappa(q)} \nabla\varphi_i\cdot \nabla\varphi_j (\xi_\kappa^\ast\cdot n){\rm d}s.$$
By an integration by parts, since the potentials $ \overline{\varphi}'_i $
are harmonic in $\mathcal F(q)$
and their normal derivatives vanish on 
$\partial\mathcal F(q) \setminus \partial\mathcal S_{\kappa}(q)$
 (see Lemma \ref{lem-banane}),
\begin{gather*}
\frac{d \mathcal M^a_{ij}}{dt}(0)
=  \int_{\partial\mathcal S_\kappa(q)} 
\Big(\frac{\partial}{\partial n}(\frac{\partial\varphi_i}{\partial \boldsymbol q_\kappa}\cdot p_\kappa^\ast) \varphi_j  
+ \varphi_i  \frac{\partial}{\partial n}(\frac{\partial\varphi_j}{\partial \boldsymbol q_\kappa}\cdot p_\kappa^\ast)
 \Big)\, {\rm d}s
+ \int_{\partial \mathcal S_\kappa(q)} \nabla\varphi_i\cdot \nabla\varphi_j (\xi_\kappa^\ast\cdot n){\rm d}s.
\end{gather*}

We now use  \eqref{def_BC} and integrate by parts on the boundary. We obtain:
\begin{align*}
\frac{d \mathcal M^a_{ij}}{dt}(0)
:= &- 2 \int_{\partial\mathcal S_\kappa(q)} \frac{\partial\varphi_i}{\partial\tau} \frac{\partial\varphi_j}{\partial\tau} (\xi_\kappa^\ast\cdot n){\rm d}s
+ \int_{\partial\mathcal S_\kappa(q)} \frac{\partial\varphi_j}{\partial\tau} (\xi_i\cdot \tau)(\xi_\kappa^\ast\cdot n){\rm d}s \\
&+ \int_{\partial\mathcal S_\kappa(q)} \frac{\partial\varphi_i}{\partial\tau} (\xi_j\cdot \tau)(\xi_\kappa^\ast\cdot n)\, {\rm d}s
+ \int_{\partial\mathcal S_\kappa(q)}  \nabla\varphi_i \cdot \nabla\varphi_j  (\xi_\kappa^\ast\cdot n)\, {\rm d}s \\
&- \delta_{3,(i)} \delta_{\kappa,[i]} \int_{\partial\mathcal S_\kappa(q)}\varphi_j (\ell_\kappa^\ast\cdot \tau){\rm d}s
- \delta_{3,(j)} \delta_{\kappa,[j]} \int_{\partial\mathcal S_\kappa(q)}\varphi_i (\ell_\kappa^\ast\cdot \tau){\rm d}s.
\end{align*}
Then, since $\frac{\partial\varphi_i}{\partial n} =\xi_i \cdot n$ and $\frac{\partial\varphi_j}{\partial n} =\xi_j \cdot n$,  we observe that 
\begin{gather*}
- 2 \frac{\partial\varphi_i}{\partial\tau} \frac{\partial\varphi_j}{\partial\tau}
+
 \nabla\varphi_i \cdot \nabla\varphi_j 
 =
 \frac{\partial\varphi_i}{\partial n} (\xi_j \cdot n)
 +
  \frac{\partial\varphi_j}{\partial n} (\xi_i \cdot n)
 -  \nabla\varphi_i \cdot \nabla\varphi_j ,
 \end{gather*}
 and then that 
 \begin{gather*}
 \frac{\partial\varphi_i}{\partial n} (\xi_j \cdot n) + \frac{\partial\varphi_i}{\partial\tau} (\xi_j\cdot \tau) = \nabla \varphi_i \cdot \xi_j   \text{ and }
 \frac{\partial\varphi_j}{\partial n} (\xi_i \cdot n) + 
\frac{\partial\varphi_j}{\partial\tau} (\xi_i\cdot \tau) 
= \nabla \varphi_j \cdot \xi_i , 
\end{gather*}
to conclude. 
\end{proof}
Using now the expressions in Lemma~\ref{express:christoff} we get Proposition~\ref{Prop-L}.
\end{proof}
%
\section{Proof of Proposition~\ref{lem-DC}}
\label{Subsec:DC}
 This section is devoted to the proof of Proposition~\ref{lem-DC}.
We first  observe that Lemma \ref{lem37-GMS} can be adapted to determine the  shape derivative of the stream function $\psi_\mu$ in terms of  the 
 shape derivatives of the  constants $C_{\mu,\nu}$  associated with the boundaries $\partial \mathcal S_\nu$, for  $\nu$ in $\{1,\dots, N\}$.
\begin{lemma}
 \label{eqqe}
For  $\kappa,\mu$ in $\{1,\dots, N\}$, for any $p_\kappa^\ast$ in $ \R^3$, 
$\left(\frac{\partial\psi_{\mu}}{\partial \boldsymbol q_\kappa}\cdot p_\kappa^\ast\right)(q,\cdot) $  is harmonic in $\mathcal F(q)$ and satisfies for  $\nu$ in $\{1,\dots, N\}$,
\begin{equation} \label{def_BBCC}
\left(\frac{\partial\psi_{\mu}}{\partial \boldsymbol q_\kappa}\cdot p_\kappa^\ast\right)(q,\cdot) =\frac {\partial C_{\mu,\nu}(q)}{\partial \boldsymbol q_\kappa}\cdot p_\kappa^\ast - \Big(\frac {\partial \psi_{\mu}}{\partial n} \frac {\partial \boldsymbol \varphi_{\kappa}}{\partial n} \Big)(q,\cdot)\cdot p_\kappa^\ast \text{  on } \partial \mathcal S_\nu .
\end{equation}
\end{lemma}
%
%
Then, 
using  the expression of $E_{\kappa}$ when $\omega=0$, see  \eqref{sec-marre}, we deduce Proposition~\ref{lem-DC} from the following lemma, whose proof relies  on 
Lemma \ref{eqqe}.
\begin{lemma}\label{express:christoff2}
For every $\mu,\nu,\kappa $ in $\{1,\ldots,N\}$, for every $p_\kappa^\ast=(\ell_\kappa^{\ast t},r_\kappa^\ast)^t$ in $\R^3$, 
\[ 
\left(\frac{\partial C_{\mu,\nu}}{\partial \boldsymbol q_\kappa}\cdot p_\kappa^\ast\right)(q)
= 
-\int_{\partial\mathcal S_{\kappa}(q)}  \frac{\partial\psi_\mu}{\partial n} \frac{\partial\psi_{\nu}}{\partial n}  (\xi_\kappa^\ast\cdot n) \,{\rm d}s ,
\]
where  $\xi_\kappa^\ast=r_\kappa^\ast(\cdot-h_\kappa)^\perp+\ell_\kappa^\ast$.
\end{lemma}
%
\begin{proof}
We first observe that  for every $\mu,\nu$ in $\{1,\ldots,N\}$,
$$C_{\mu,\nu}(q)=-\int_{\mathcal F(q)}\nabla{\psi}_\mu\cdot \nabla{\psi}_\nu\, {\rm d}x .$$
By Reynold's transport theorem,   for every $\kappa$ in $\{1,\ldots,N\}$, for every $p_\kappa^\ast=(\ell_\kappa^{\ast t},r_\kappa^\ast)^t$ in $\R^3$, 
$$
\left(\frac{\partial C_{\mu,\nu}}{\partial \boldsymbol q_\kappa}\cdot p_\kappa^\ast\right)(q)
=- \int_{\mathcal F(q)} (\nabla (\frac{\partial\psi_{\mu}}{\partial \boldsymbol q_\kappa}\cdot p_\kappa^\ast)\cdot \nabla{\psi}_\nu
+ \nabla{\psi}_\mu\cdot \nabla (\frac{\partial\psi_{\nu}}{\partial \boldsymbol q_\kappa}\cdot p_\kappa^\ast))\, {\rm d}x 
- \int_{\partial \mathcal S_\kappa(q)} \nabla\psi_\mu\cdot \nabla\psi_\nu (\xi_\kappa^\ast\cdot n){\rm d}s,$$
where  $\xi_\kappa^\ast=r_\kappa^\ast(\cdot-h_\kappa)^\perp+\ell_\kappa^\ast$.
By integrations by parts and Lemma  \ref{eqqe}, since  the first term in the  right hand side term of \eqref{def_BBCC} is constant, 
\begin{align*}
\left(\frac{\partial C_{\mu,\nu}}{\partial \boldsymbol q_\kappa}\cdot p_\kappa^\ast\right)(q)
= & 2  \left(\frac{\partial C_{\mu,\nu}}{\partial \boldsymbol q_\kappa}\cdot p_\kappa^\ast\right)(q)
+ \int_{\partial\mathcal S_\kappa(q)}  \Big(\frac {\partial \psi_{\mu}}{\partial n} \frac {\partial \boldsymbol \varphi_{\kappa}}{\partial n} \Big)(q,\cdot)\cdot p_\kappa^\ast  \frac{\partial\psi_\nu}{\partial n} \,{\rm d}s \\
&+ \int_{\partial\mathcal S_\kappa(q)}  \Big(\frac {\partial \psi_{\nu}}{\partial n} \frac {\partial \boldsymbol \varphi_{\kappa}}{\partial n} \Big)(q,\cdot)\cdot p_\kappa^\ast  \frac{\partial\psi_\mu}{\partial n} \,{\rm d}s
- \int_{\partial\mathcal S_\kappa(q)} \frac{\partial\psi_\mu}{\partial n} \frac{\partial\psi_{\nu}}{\partial n}(\xi_\kappa^\ast\cdot n)\, {\rm d}s ,
\end{align*}
%
%
%
%
and, since  $ \frac{\partial\boldsymbol\varphi_{\kappa}}{\partial n} \cdot p_\kappa^\ast =  \xi_\kappa^\ast \cdot n$ on $\partial\mathcal S_\kappa(q)$, 
the conclusion follows.
\end{proof}
%

\par
\bigskip
\par
\noindent
{\bf Acknowledgements.} The authors are partially supported by the Agence Nationale de la Recherche, Project IFSMACS, grant ANR-15-CE40-0010.
The first, second and last authors are partially supported by the Agence Nationale de la Recherche,  Project SINGFLOWS grant ANR-18-CE40-0027-01. 
The fourth author is partially supported by the Agence Nationale de la Recherche, Project BORDS, grant ANR-16-CE40-0027-01, the Conseil R\'egionale d'Aquitaine, grant 2015.1047.CP, by the H2020-MSCA-ITN-2017 program, Project ConFlex, Grant ETN-765579. 
%
%
%
%
%
\def\cprime{$'$}

\end{document}